\documentclass[12pt,a4paper]{amsart}
\usepackage{enumerate}
\usepackage{graphicx}
\usepackage{amsmath,verbatim}
\usepackage{amssymb}
\usepackage{pgf,tikz}
\usetikzlibrary{arrows}
\usepackage[utf8]{inputenc}
\usepackage{hyperref}

\theoremstyle{plain}
\newtheorem{thm}{Theorem}[section]
\newtheorem{thmx}{Theorem}

\newtheorem{lemma}[thm]{Lemma}
\newtheorem{corollary}[thm]{Corollary}

\newtoks\prt

\theoremstyle{definition}

\newtheorem{definition}[thm]{Definition}

\def\eqn#1$$#2$${\begin{equation}\label#1#2\end{equation}}

\numberwithin{equation}{section}

\definecolor{ffqqqq}{rgb}{1.,0.,0.}
\definecolor{xfqqff}{rgb}{0.4980392156862745,0.,1.}
\definecolor{qqqqff}{rgb}{0.,0.,1.}

\def\loc{\operatorname{loc}}
\def\co{\operatorname{co}}

\def\T{\mathcal{T}}

\def\phi{\varphi}
\def\epsilon{\varepsilon}

\def\er{\mathbb R}

\def\id{\operatorname{id}}
\def\sgn{\operatorname{sgn}}
\def\rn{\mathbb R^n}
\def\sgn{\operatorname{sgn}}

\newtoks\by
\newtoks\paper
\newtoks\book
\newtoks\jour
\newtoks\yr
\newtoks\pages
\newtoks\vol
\newtoks\publ
\def\ota{{\hbox\vol{???}}}
\def\cLear{\by=\ota\paper=\ota\book=\ota\jour=\ota\yr=\ota
	\pages=\ota\vol=\ota\publ=\ota}
\def\endpaper{\the\by, {\the\paper},
	\textit{\the\jour} \textbf{\the\vol} (\the\yr), \the\pages.\cLear}
\def\endbook{\the\by, \textit{\the\book}, \the\publ.\cLear}
\def\endprep{\the\by, \textit{\the\paper}, \the\jour.\cLear}
\def\endyearprep{\the\by, \textit{\the\paper}, \the\jour, (\the\yr).\cLear}
\def\name#1#2{#2 #1}
\def\nom{ \rm no. }

\MakeRobust{\ref}

\makeatletter
\newcommand{\labeltext}[2]{%
	\@bsphack
	\def\@currentlabel{#1}{\label{#2}}%
	\@esphack
}
\makeatother

\def\step#1#2#3{\par \noindent{{\\ \bf Step~\labeltext{#1}{#3}#1. }{\bf #2. }}}

\begin{document}
	
	\title[Approximation of p-w affine homeomorphisms]{Diffeomorphic approximation of piecewise affine homeomorphisms}
	
	\author[D. Campbell]{Daniel Campbell}
	\address{Department of Mathematical Analysis, Charles University, So\-ko\-lovsk\'a 83, 186~00 Prague 8, Czech Republic}
	\email{\tt campbell@karlin.mff.cuni.cz}
	
	\author[L. D'Onofrio]{Luigi D'Onofrio}
	\address{Dipartimento di Scienze e Tecnologie Universita' di Napoli Parthenope, ISOLA C4, 80100 Napoli, Italy}
	\email{\tt luigi.donofrio@uniparthenope.it}
	
	\author[T. Vítek]{Tomáš Vítek}
	\address{Mathematical Institute, Faculty of Mathematics and Physics, Charles University, So\-ko\-lovsk\'a 83, 186~00 Prague 8, Czech Republic}
	\email{\tt tomas.vitek.university@gmail.com}

	\thanks{The first author was supported by the grant GA\v{C}R P201/24-10505S and by the Ministry of Education, Youth and Sport of
		the Czech Republic grant number LL2105 CONTACT. The second author was supported by GNAMPA and by the European Union - NextGenerationEU within the framework of PNRR  Mission 4 - ``PRIN 2022'' - grant number 2022BCFHN2 - Advanced theoretical aspects in PDEs and their applications CUP I53D23002300006, the third author was supported by the Charles University Research Center program No. UNCE/24/SCI/022 and the project SVV-2025-260837-.}

	
	\begin{abstract}
		Given any $f$, a locally finitely piecewise affine homeomorphism of $\Omega \subset \mathbb{R}^d$ onto $\Delta \subset \mathbb{R}^d$ (for $d=3, 4$) such that $f\in W^{1,p}(\Omega, \mathbb{R}^d)$ and $f^{-1}\in W^{1,q}(\Delta, \mathbb{R}^d)$, $1\leq p ,q < \infty$ and any $\epsilon >0$ we construct a diffeomorphism $\tilde{f}$ such that 
		$$\|f-\tilde{f}\|_{W^{1,p}(\Omega,\mathbb{R}^d)} + \|f^{-1}-\tilde{f}^{-1}\|_{W^{1,q}(\Delta,\mathbb{R}^d)} < \epsilon.$$
	\end{abstract}
	
	\maketitle
	
	\section{Introduction}
	
	The approximation of homeomorphisms by diffeomorphisms represents a fundamental problem in geometric analysis and differential topology. In particular, understanding when and how homeomorphisms with limited regularity can be approximated by piecewise linear homeomorphisms, smooth homeomorphisms or diffeomorphisms - while preserving key analytic properties - has significant implications across several areas of mathematics including geometric function theory, partial differential equations, and continuum mechanics. A summary of open questions on these topics can be found in \cite{Hencl}.
	
	In this paper, we provide an answer to \cite[Open Problem 2]{Hencl} in dimension $3$ and $4$. We study locally finitely piecewise affine homeomorphisms $f: \Omega \to \Delta$,  where $\Omega, \Delta \subset \mathbb{R}^d$, $d=3,4$ (for the definition see Definition~\ref{SimplexDef}). The key role these maps play in several areas makes them natural objects of study as they bridge the gap between purely topological and differentiable structures. Our main result is the following:
	
	\begin{thmx}\label{Core}
		Let $d=3$ or $d=4$. Let $\Omega\subset\er^d$ be a domain and let $p,q\in[1,\infty)$. Let $f:\Omega\to\er^d$ be a locally finite piecewise affine homeomorphism. Then, for every $\epsilon>0$, there exists a $\mathcal{C}^{\infty}$-diffeomorphism $\tilde{f}$ satisfying 
		$$
		\|\tilde{f} - f\|_{L^{\infty}(\Omega)}  + \|\tilde{f}^{-1} - f^{-1}\|_{L^{\infty}(f(\Omega))} < \epsilon
		$$
		and
		$$
		\|D\tilde{f} - Df\|_{L^p(\Omega)}  + \|D\tilde{f}^{-1} - Df^{-1}\|_{L^q(f(\Omega))}< \epsilon.
		$$
	\end{thmx}

	Recently a body of research has accrued around the approximation of homeomorphisms by diffeomorphisms in Sobolev spaces. Let us note some of the most important results. The celebrated breakthrough results in the area, which stimulated much interest in the subject, were given by Iwaniec, Kovalev, and Onninen in \cite{IKO} and \cite{IKO2}, where they found diffeomorphic approximations to any planar homeomorphism $f\in W^{1,p}(\Omega,\mathbb{R}^2)$, for any $1<p<\infty$ in the $W^{1,p}$ norm. The remaining open case (in the plane), $p=1$, was solved by Hencl and Pratelli in \cite{HP} using a different method: first generating piecewise affine approximations and then smoothing them.
	
	All previous results which dealing with the approximation of homeomorpshisms by diffeomorphisms are planar in nature. To date, there are no known results (to the best of the authors' knowledge) allowing approximation of Sobolev homeomorphisms by diffeomorphisms in the Sobolev space in higher dimensions. The construction in \cite{IKO} yields a diffeomorphism, whereas the construction in \cite{HP} generates a piecewise affine homeomorphism. Therefore, we quickly arrive at the question of whether approximation by piecewise affine homeomorphisms or by diffeomorphisms is equivalent. One implication of this question was shown in \cite{IO2}. Here they show that any diffeomorphism can be triangulated for any dimension $d\geq 1$. The maps generated by their process remain close to the original diffeomorphism in $W^{1,\infty}$ and the same holds for the inverse. On the other hand, it is impossible for a diffeomorphism to be arbitrarily close to a piecewise affine homeomorphism in $W^{1,\infty}$ since the derivative must be continuous and this adds some extra difficulty to the opposite implication.
	
	Known results in this direction, so far, are as follows. In \cite{MP} it was proven that in the planar case one can approximate a Sobolev piecewise affine map by diffeomorphisms in the Sobolev space. A further planar result working also with the second derivatives (of the map but not its inverse) has been achieved in \cite{CH}. It was proven in \cite{CS} that one can approximate any Sobolev piecewise affine homeomorphism by smooth homeomorphisms (this does not require that the approximations have smooth inverse) from $\er^d$ to $\er^d$ for any $d$.
	
	One disadvantage of the previously mentioned result is that it is known that concepts like connectedness and closure are quite different when we require only smooth homeomorphisms as opposed to diffeomorphisms, we refer the reader to \cite{Rob2}. Moreover, the Theorem in \cite{Rob2} implies that diffeomorphisms are connected in pointwise $\mathcal{C}^k$ topology, but this is not true in $\mathcal{C}^k$ uniform topology for all dimensions as can be observed from \cite[Section II]{Hatcher50} and \cite[Table 1 and 2]{StableHomotopyGroups} and connectedness of diffeomorphisms is a critical part of our construction. In light of the technique in \cite{HP}, the question of the equivalence of approximation by piecewise affine homeomorphisms and diffeomorphisms is very interesting in higher dimension and our result covers precisely this gap in dimension 3 and 4.
	
	Theorem~\ref{Core} can be generalized to the more general case in the setting of rearrangement invariant function spaces. For a definition of a rearrangement invariant function space $X$ with absolutely continuous norm, see Section~\ref{Wales}.
	
	\begin{thmx}\label{main}
		Let $d=3$ or $d=4$. Let $\Omega\subset\er^d$ be a domain. Let $f:\Omega\to\er^d$ be a locally finite piecewise affine homeomorphism. Let $X(\Omega), Y(f(\Omega))$ be a pair of rearrangement invariant Banach function spaces with $\phi_X(t), \phi_Y(t) \to 0$ as $t\to 0$. Then for every $\epsilon>0$ there exists a diffeomorphism $\tilde{f} \in \mathcal{C}^{\infty}(\Omega, \er^d)$ with 
		$$
		\|\tilde{f} - f\|_{L^{\infty}(\Omega)}  + \|\tilde{f}^{-1} - f^{-1}\|_{L^{\infty}(f(\Omega))} < \epsilon
		$$
		and
		$$
		\|D\tilde{f} - Df\|_{X(\Omega)} + \|D\tilde{f}^{-1} - Df^{-1}\|_{Y(f(\Omega))} < \epsilon.
		$$
	\end{thmx}
	
	Our approach was developed independently of the one used in \cite{Munk}. Nevertheless, it exhibits many similarities in both techniques and intermediate results. Furthermore, our result could, in principle, be extended to dimensions higher than four by restricting to a smaller class of piecewise affine homeomorphisms, as is done in that work. However, we have not pursued this direction in the present paper.

	In Section~\ref{Prelim} we introduce some notion and notation we need through the paper. Later, in Section~\ref{Meat} we give a full detailed proof of the claim in dimension three. Finally, we explain the necessary modifications for dimension four and comment on why we cannot directly generalize our approach to higher dimension.

	\section{Preliminaries}~\label{Prelim}
	In this section, we introduce some notions and notations that we use repeatedly. When we say that a map is smooth we mean of class $\mathcal{C}^{\infty}$ on its given domain.
	
	Let $u:\rn \to \er^d$ be a map. By $\partial_ku(x)$ we denote the classical partial derivative of u in direction $e_k$ and more generally, for a vector $v\in\er^3$ by $\partial_v f$ we denote the pointwise directional derivative of $f$ in the direction $v$. By $Du(x)$ we denote the pointwise gradient of $u$ at $x$. In this paper we work with locally Lipschitz maps (smooth maps, piecewise affine homeomorphisms and diffeomorphisms), therefore the pointwise partial derivatives and the weak derivatives of $u$ coincide in the $L^{1}_{\loc}$ sense.
	
	Let $\eta\in\mathcal{C}^\infty(\er)$ with 
	\begin{equation}\label{Beef}
		\eta(t) = 0 \text{ for all } t\leq 0, \, \eta(t)= 1 \text{ for all } t\geq 1 \text{ and } 0\leq \eta'\leq 2.
	\end{equation}
	\begin{definition}\label{SimplexDef}
		Let $A_1, A_2, \dots A_{n+1} \in  \rn$ be points so that the vectors $A_2 - A_1$, $A_3-A_1, \dots A_{n+1}-A_1$ are linearly independent. We call $T = \co \{A_1, A_2, \dots, A_{n+1}\}$, the convex hull of the points $A_1, A_2, \dots, A_{n+1}$ a simplex (or more specifically an $n$-simplex).
		
		Let $\{B_1, B_2, \dots, B_{k+1}\} \subset \{A_1, A_2, \dots A_{n+1}\}$, then we refer to $F =\co \{B_1, B_2, \dots, B_{k+1}\}$ as a $k$-subsimplex of $T$. Given $F$, a $k$-subsimplex of $T$ we denote by $N_F$ the space of vectors perpendicular to the vectors $B_2-B_1$, $B_3-B_1, \dots, B_{k+1}- B_1$.
		
		Let $\Omega\subset \rn$, by the term simplicial complex of $\Omega$ we refer to a countable (possibly finite) set $\T = \{T_1, T_2, \dots\}$ of simplices such that every $k$-subsimplex of a simplex from $\T$ is also in $\T$, every non-empty intersection of any two simplices in $\T$ is a $k$-subsimplex of both of them and $\Omega = \bigcup_{i} T_i$.
		
		We call a simplicial complex locally finite if for any $x\in \Omega$ there exists an $r>0$ such that only a finite number of simplices $T_i$ intersect $B(x,r)$.
		
		We refer to $f$ as a piecewise affine map on $\Omega$ if there exists a simplicial complex $\T$ on $\Omega$ such that the restriction of $f$ to each $T_i \in\T$ is equal to some affine map $L_i(x) = M_ix + c_i$. Further we say that $f$ is a locally finite piecewise affine homeomorphism, if $\T$ is locally finite.
	\end{definition}
	
	By $\pi_3:\er^3 \to \er^3$ we denote the projection
	$$
	\pi_3(x,y,z) = (x,y,0)
	$$
	and similarly by $\tilde{\pi}_3:\er^3 \to \er^2$ we denote the projection
	$$
	\tilde{\pi}_3(x,y,z) = (x,y).
	$$
	
	The following theorem follows from \cite[Theorem 6]{S}.
	
	\begin{thm}\label{Smale}
		For every $\alpha_0$, a sense-preserving  $\mathcal{C}^{\infty}$ diffeomorphism of $S^2 = \partial B_{\er^3}(0,1)$ onto itself, there exists an $\alpha\in \mathcal{C}^{\infty}([0,1]\times S^2,S^2)$ such that $\alpha(t,\cdot)$ is a sense-preserving diffeomorphism of $S^2$ for each $t$, $\alpha(0,\cdot) = \id(\cdot)$ and $\alpha(1,\cdot) = \alpha_0(\cdot)$.
	\end{thm}
	\begin{corollary}\label{VertThm}
		Let $\tilde{f}$ be a $\mathcal{C}^{\infty}$ sense-preserving diffeomorphism of $S^2 = \partial B_{\er^3}(0,1)$ onto itself. There exists $g$, a $\mathcal{C}^\infty$ diffeomorphism of $\overline{B_{\er^3}(0,1)}$ onto $\overline{B_{\er^3}(0,1)}$ with $g_{\rceil \partial B_{\er^3}(0,1)} = \tilde{f}$.
	\end{corollary}
	\begin{proof}
		By Theorem~\ref{Smale} we have an isotopy $h:\partial B(0,1)\times [0,1] \to \partial B(0,1)$ such that $h(x,t)$, $t\in [0,1]$, $h(x,1) = \tilde f(x)$ and $h(x,0) = x$ for all $x\in \partial B(0,1)$. In fact may assume that $h(x,t) =  \tilde f(x)$ for $t\in [\tfrac{2}{3},1]$ and $h(x,t) =  x$ for $t\in [0,\tfrac{1}{3}$]. Then we define $g:\overline{B(0,1)} \to \overline{B(0,1)}$ as 
		$$
		g(x) = \begin{cases}
			|x|h\Big(\frac{x}{|x|}, |x|\Big) \quad & |x|>0\\
			0 &x=0.
		\end{cases}
		$$
		Then $g$ is a diffeomorphism and $g = \tilde{f}$ on $\partial B(0,1)$.
	\end{proof}
	
	The following lemma is well-known, we include it for simple reference.
	\begin{lemma}\label{Book}
		Let $d\geq 1$, let $\Omega, \Delta \subset \mathbb{R}^d$ be bounded domains and let $\mu:\overline{\Omega}\to \mathbb{R}^d$ be a local diffeomorphism of $\Omega\subset \mathbb{R}^d$ and continuous on $\overline{\Omega}$. Further we assume that there is a neighborhood $O$ of $\partial\Omega$ in $\overline{\Omega}$ and a neighborhood $D$ of $\partial\Delta$ in $\overline{\Delta}$ such that $\mu$ is a homeomorphism $O$ onto $D$. Then $\mu$ is a diffeomorphism of $\Omega$ onto $\Delta$. 
	\end{lemma}
	\begin{proof}
		It is a well known fact from degree theory that $\deg(\mu, \cdot, \Omega)$ is constant on components of $\rn\setminus \mu(\partial \Omega)$. Also it is well known that the degree of a homeomorphism is $\pm 1$.  Since $\mu$ is equal to a homeomorphism on $D$, we have that $\deg(\mu,\cdot,  \Omega) =\pm 1$ on $\Delta$ and is constant. Since $\mu$ is a local diffeomorphism, the sign of $\det D\mu$ is constant (and equal to $\deg(\mu, y,\Omega)$ for every $y\in \Delta$) on $\Omega$. Since every point $y\in \Delta$ is a regular value of $\mu$, there is at most one point $x\in \Omega$ such that $\mu(x) = y$. Similarly for all $y\in \mathbb{R}^d\setminus \overline{\Omega}$ it holds that $\mu^{-1}(\{y\})= \emptyset$. Thus $\mu$ is injective on $\Omega$ and $\mu(\Omega) = \Delta$. It follows that $\mu$ is a diffeomorphism of $\Omega$ onto $\Delta$.
	\end{proof}
	
	Let $\mu:S^{d-1} \to S^{d-1}$ be a smooth mapping. We call the derivative of $\mu$, the linear map $d\mu_x:T_xS^{d-1}\to T_{\mu(x)}S^{d-1}$ defined in the usual way. Let $y\in S^{d-1}$ be such that the determinant of $d\mu_x$ is non-zero for all $x\in \mu^{-1}(y)$ (which is true for almost all points $y\in \mu(S^{d-1})$) then we define
	$$\deg(\mu,y):=\sum_{x\in \mu^{-1}(y)}\sgn \det d\mu_x.$$
	For more details we refer the reader to \cite[Chapter~5]{Milnor}.
	
	\begin{thm}\label{AlgTop}
		Let $d\geq 3$ and $\mu:S^{d-1} \to S^{d-1}$ be a local diffeomorphism, then $\mu$ is a diffeomorphism.
	\end{thm}
	
	\begin{proof}
		Since $\mu$ is a local diffeomorphism and $S^{d-1}$ is compact and connected, it is a covering map. And since $S^{d-1}$ is simply connected, $\mu$ is a homeomorphism by the uniqueness of the universal covering, see \cite[Proposition 1.37]{AlgTopHatcher}.
	\end{proof}
	
	
	
	\subsection{Rearrangement invariant function spaces}\label{Wales}
	
	Let $E$ be a (non-negligible) Lebesgue measurable subset of $\mathbb R^d$ of finite measure. We denote by $\mathcal{L}^d(E)$ its Lebesgue measure. We set 
	$$
	L^0(E)=\left\{ f: f \text{ is measurable function on } E \text{ with values in }[-\infty, + \infty] \right\}
	$$
	and
	$$L^0_+(E)= \left\{ f\in L^0(E): f\geq 0\right\}.$$
	We identify functions $f_1,f_2$ for which $\mathcal{L}^d(\{f_1\neq f_2\})=0$ in the space $L^0$.
	The non-increasing rearrangement $f^*\colon [0, +\infty] \rightarrow[0, +\infty]$ of a function $f\in L^0(E)$ is defined by
	$$
	f^*(s)=\inf\left\{ t\geq 0: \mathcal{L}^d\big (\left\{x\in E: |f(x)|>t  \right\}\big) \leq s \right\} \qquad s\in [0, + \infty)
	$$
	and the \textit{Hardy-Littlewood maximal function} $f^{**}:(0,\mathcal{L}^d(E))\to [0,+\infty)$ is given by
	$$
	f^{**}(t):=\frac{1}{t}\int_0^t f^*(s)\, dx s.
	$$
	
	\begin{definition}\label{RIS}
		Let $E$ be a Lebesgue measurable subset of $\mathbb R^d$ and let $\| \cdot \|_{X(E)}: L^0(E) \rightarrow [0, +\infty]$ be a functional. Consider the following properties
		\begin{enumerate}[\upshape(P1)]
			\item the set $\{u\in L^0(E): \|u\|_{X(E)} < \infty \}$ is a linear space and $\|\cdot\|_{X(E)}$ is a norm on this space.\label{primo}
			\item For all $f,g\in L^0_+(E)$ the inequality $f(x)\leq g(x)$  for a.e. $x$ in $E$ implies $\| f\|_{X(E)}\leq \|g\|_{X(E)}$. 
			\item $\sup\limits_{k}\| f_k\|_{X(E)}= \| f\|_{X(E)}$ if $0\leq f_k(x) \nearrow f(x)$ for a.e. $x$ in E.
			
			\item Let $G\subset E$ be a set of a finite measure. Then
			$$
			\|\chi_G\|_{X(E)}<\infty.
			$$
			\item Let $G\subset E$ be a set of a finite measure. Then there exists a constant $C_G$ depending only on the choice of the set $G$ for which
			$$
			\|f\chi_G\|_{L^1(E)}\leq C_G\|f\chi_G\|_{X(E)}.
			$$
			for all $f\in L^0(E)$.
			
			\item \label{ultimo}$$
			\| f\|_{X(E)}= \|g\|_{X(E)}\text{ whenever } f^*= g^* \text{\emph{(rearrangement invariance)}}
			$$
		\end{enumerate}
		If $\|\cdot\|_X$ enjoys the properties (P1)-(P5) we call it a Banach function norm. If it also enjoys (P6) we call it a rearrangement invariant Banach function norm. Let $\|\cdot\|_X$ be (rearrangement-invariant) Banach function norm then we call set
		$$
		X(E):=\{f\in L^0(E):\|f\|_{X(E)}<\infty\}
		$$
		endowed with the norm $\|\cdot\|_{X(E)}$ a (rearrangement-invariant) Banach function space.
	\end{definition}
	
	Given a r.i. Banach function space $X(E)$ and $0\leq s < \mathcal{L}^d(E)$ one may define the \textit{fundamental function of $X(E)$} by
	\begin{equation}\label{Jihad}
		\varphi_{X(E)}(s):= \| \chi_G\|_{X(E)}
	\end{equation}
	where $G\subset E$ is an arbitrary subset of $E$ of measure $s$.
	
	Let $X(E)$ be a Banach function space. We say that $X(E)$ has \textit{locally absolutely continuous norm} if for any finite measure set $M\subset E$ and any function $f\in X(E)$ one has 
	$$
	M\supset M_n \textup{ such that } \mathcal L^{d}(M_n)\rightarrow 0\quad\textup{implies}\quad \|f\chi_{M_n}\|_X\rightarrow 0. 
	$$
	
	The following lemma is from \cite{CGSS}.
	\begin{lemma}\label{Rozumny}
		Let $G\subset\mathbb{R}^n$ be a set of finite measure and $X(G)$ be a rearrangement invariant Banach function space satisfying
		\begin{equation}\label{lim=0}
			\displaystyle{\lim_{t\rightarrow 0+}}\varphi_X(t)=0.
		\end{equation}
		Then for every $M >0$ and $\tilde{\epsilon} > 0$ there exists a $\tilde{\delta}>0$ such that for all $u\in X(G)$ with $\|u\|_{L^{\infty}(G)} \leq M$ and $\|u\|_{L^{1}(G)}<\tilde{\delta}  \mathcal{L}^d(G)$ one has $\|u\|_{X(G)}< \tilde{\epsilon}  \mathcal{L}^d(G)$.
	\end{lemma}
	
	\section{The construction in 3D}\label{Meat}
	In this section, we present a detailed proof of our result in dimension $d=3$. We begin with the following lemma that introduces a map that we use repeatedly through the paper. In our applications of the lemma, we typically ensure a priori that $\sigma <1$. 
	\begin{lemma}\label{FaceLem}
		Let the affine mappings $A_1, A_2 : \mathbb{R}^3 \to \mathbb{R}^3$ be such that the mapping
		$$
		f(x)=\begin{cases}
			A_1(x) \quad &\text{for }x_1\leq0\\
			A_2(x) \quad &\text{for }x_1\geq 0.
		\end{cases}
		$$
		is a sense-preserving homeomorphism. Further, let $[\partial_1 A_2]^1 \geq [\partial_1 A_1]^1>0$ and $w\in \mathcal{C}^{\infty}(\er^2)$ be a bounded positive function. Then, there exists a $\sigma>0$ such that whenever $|Dw(y,z)|\leq \sigma$ the map 
		\begin{equation}\label{defg}
			g_w(x) = \Big(1-\eta\Big(\frac{x_1}{w(x_2,x_3)}\Big)\Big)A_1(x) + \eta\Big(\frac{x_1}{w(x_2,x_3)}\Big)A_2(x)
		\end{equation}
		is a diffeomorphism $g_w: \mathbb{R}^3 \to \mathbb{R}^3$ satisfying 
		\begin{equation}\label{NoChange}
			g_w(x_1,x_2,x_3) = f(x_1,x_2,x_3)
		\end{equation}
		on the set
		$$
		\{(x_1,x_2,x_3)\in \er^3; x_1\leq 0\}\cup \{(x_1,x_2,x_3)\in \er^3; x_1\geq w(x_2,x_3)\}.
		$$
		Further, there exists an absolute constant $C_{\ref{Guessed}}$ independent of $f$ and $w$ such that 
		\begin{equation}\label{Guessed}
			\|Dg_w\|_{L^\infty}\leq C_{\ref{Guessed}}(1+\sigma)\|Df\|_{L^\infty}.
		\end{equation}
	\end{lemma}
	\begin{proof}
		Since $A_1$, $A_2$, $w$ and $\eta$ are all in $\mathcal{C}^{\infty}$ and $w>0$ everywhere we have that $g_w\in \mathcal{C}^{\infty}(\er^3)$. Further, from the definition, it is immediately obvious that $g_w(x_1,x_2,x_3) = f(x_1,x_2,x_3)$ on the set $\{x_1\leq 0\}\cup \{x_1\geq w(x_2,x_3)\}$.

		Note that, since $A_1 = A_2$ on the plane $\{x\in \er^3; x_1=0\}$, we have $\partial_2A_1 = \partial_2A_2$ and $\partial_3A_1 = \partial_3A_2$. The definition of $g_w$ is independent of addition by a vector, i.e., let $v\in \er^3$ and let $\tilde{f} = f+v$, then $\tilde{g}_{w} =g_w + v$. The same holds regarding the action of linear isometries, $O(3)$. Therefore, without loss of generality, we can assume that $f$ maps the plane $\{x\in \er^3: x_1=0\}$ onto itself, while still having positive Jacobian and having $[\partial_1 A_2]^1 \geq [\partial_1 A_1]^1>0$. Our choice of isometry implies that $[\partial_2 A_1]^1 = [\partial_3 A_1]^1 = [\partial_2 A_2]^1 = [\partial_3 A_2]^1 = 0$.

		Since we have that
		\begin{equation}\label{TheMatrix}
			Dg_w = \left( \begin{matrix}
			[\partial_1g_w]^1 &0 &0\\
			[\partial_1g_w]^2 &[\partial_2g_w]^2 & [\partial_3g_w]^2\\
			[\partial_1g_w]^3 &[\partial_2g_w]^3 & [\partial_3g_w]^3
		\end{matrix}\right).
		\end{equation}
		we calculate the Jacobian of $g_w$ by an expansion of the first row. We have
		\begin{equation}\label{D1Posi}
			\begin{aligned}[]
				[\partial_1g_w]^1  =&  \Big(1-\eta\Big(\frac{x_1}{w(x_2,x_3)}\Big)\Big)[\partial_1 A_1]^1+ \eta\Big(\frac{x_1}{w(x_2,x_3)}\Big)[\partial_1 A_2]^1\\
				&  \quad +\frac{1}{w(x_2,x_3)}\eta'\Big(\frac{x_1}{w(x_2,x_3)}\Big)\big([A_2(x)]^1 - [A_1(x)]^1\big) \\
				\geq&  [\partial_1 A_1]^1 + \frac{1}{w(x_2,x_3)}\eta'\Big(\frac{x_1}{w(x_2,x_3)}\Big)\big([A_2(x)]^1 - [A_1(x)]^1\big) \\
				\geq&  [\partial_1 A_1]^1 >0.
			\end{aligned}
		\end{equation}
		Now we want to show two things. Firstly, that 
		$$
		\det \left( \begin{matrix}
			[\partial_2g_w]^2 & [\partial_3g_w]^2\\
			[\partial_2g_w]^3 & [\partial_3g_w]^3
		\end{matrix}\right)>0
		$$
		implying that $J_{g_w}>0$ everywhere. Secondly we want to prove that $g_w$ is equal to a homeomorphism outside of some large ball. Then we can finish the proof by applying Lemma~\ref{Book}.

		Regarding the first point, note that at every $x$, we have $\partial_j A_1(x) = \partial_j A_2(x) = \partial_j f(x)$ for $ j = 2, 3$. Further, it holds that
		$$
			\begin{aligned}
				|\partial_jg_w(x) - \partial_jA_1| &\leq \partial_jw(x_2,x_3) \frac{x_1|A_2(x) - A_1(x)|}{w^2(x_2,x_3)}\eta'\Big(\frac{x_1}{w(x_2,x_3)}\Big).
			\end{aligned}
			$$
		Either $\eta'\Big(\frac{x_1}{w(x_2,x_3)}\Big) = 0$, or $x_1<w(x_2,x_3)$ implying that $|A_2(x) - A_1(x)|< 2\|Df\|_{\infty}w(x_2,x_3)$. Therefore, we have
		\begin{equation}\label{EstimateDJ}
				\begin{aligned}
				|\partial_jg_w(x) - \partial_jA_1|\leq 2\|Df\|_{\infty}|Dw|.
			\end{aligned}
		\end{equation}
		The above implies that
		\begin{equation}\label{TruEstimateDJ}
				\|\partial_jg_w(x)\|_{\infty}\leq C\|Df\|_{\infty}(1+\|Dw\|_{\infty})
		\end{equation}
		for an absolute constant $C$. By the continuity of the determinant on entries, it follows, for a sufficiently small choice of $\sigma$ and $|Dw|\leq \sigma$, that 
		\begin{equation}\label{Poker}
			\det \left( \begin{matrix}
			[\partial_2g_w]^2 & [\partial_3g_w]^2\\
			[\partial_2g_w]^3 & [\partial_3g_w]^3
		\end{matrix}\right) \geq \tfrac{1}{2}\det \left( \begin{matrix}
		[\partial_2f]^2 & [\partial_3f]^2\\
		[\partial_2f]^3 & [\partial_3f]^3
		\end{matrix}\right)>0.
		\end{equation}
		Thus we have concluded that $J_{g_w}>0$ everywhere.
		
		Instead of showing that $g_{w}$ is injective outside a large ball we define a modified $\tilde{w}$ equal to $w$ close to the origin  and show that the modified is injective outside a large ball. Let $r>0$ be given and let
		$$
			R:=r+2+\frac{\sup\{|w|\}+1}{\sigma}.
		$$
		Then we can define a smooth $\tilde{w}$ such that $\tilde{w} = w$ on $B(0,r)\subset \er^2$, $w= 1$ outside $B(0,R)$ and $|D\tilde{w}|\leq \sigma$. To achieve this let us firstly define $\tilde{\tilde{w}} := w$ on $B(0,r+1)\subset \er^2$, $w= 1$ outside $B(0,R-1)$ and extend this map onto $\er^3$ by the McShame extension. We define $\tilde{w}$ as the standard mollification of $\tilde{w}$. Since $g_w = g_{\tilde{w}}$ on $\er\times B(0,r)$ it suffices to show that $g_{\tilde{w}}$ is a diffeomorphism. Since we already know that the smooth map $g_{\tilde{w}}$ has positive Jacobian everywhere and is therefore a local diffeomorphism, it suffices to show that $g_{\tilde{w}}$ is injective outside $Q(0,R)\subset \er^3$.
		
		Since $R > |w|$ we have on the set $\big(\er^3\setminus Q_{\er^3}(0,R)\big) \cap \{x_1>1\}$ that $g_{\tilde w} = f$ and $f$ is injective. On the set $\er\times(\er^2\setminus Q_{\er^2}(0,R))$ we have, by \eqref{TheMatrix} and  \eqref{D1Posi}, that $g_{\tilde{w}}^1(x) = g_{\tilde{w}}^1(y)$ if and only if $x_1 = y_1$. Now, thanks to the fact that $\tilde{w}(x_2,x_3) =1$ for all $(x_2,x_3)\in \er^2\setminus Q_{\er^2}(0,R)$ we have
		\begin{equation}\label{Sunlight}
			\partial_2 g_w = \partial_2f \qquad \text{and}\qquad \partial_3 g_w = \partial_3f.
		\end{equation}
		The above and the injectivity of $f$ implies that the restriction of $g_{\tilde{w}}$ to hyperplanes of type $\{x\in \er^3; x_1 = c\}$ is injective. The combination of these facts imply that $g_{\tilde{w}}$ is injective on the set $\er\times(\er^2\setminus Q_{\er^2}(0,R))$. In total we see that $g_{\tilde{w}}$ is injective on $\er^3\setminus Q_{\er^3}(0,R)$. It follows from Lemma~\ref{Book} that $g_{\tilde{w}}$ is injective on $\er^3$, especially it is injective on $B(0,r)$ and so $g_{w}$ is a diffeomorphism of $B(0,r)$. On the other hand, $r>0$ was chosen arbitrarily and therefore $g_w$ is a diffeomorphism on $\er^3$. 
		
		We estimate $\partial_1g_w$ as follows;
		$$
		|\partial_1g_w| \leq  |\partial_1 A_1|+ |\partial_1 A_2| +\eta'\Big(\frac{x_1}{w(x_2,x_3)}\Big)\frac{|A_2(x) - A_1(x)|}{w(x_2,x_3)}.
		$$
		As pointed out before \eqref{EstimateDJ}, we either have $\eta' = 0$ or $|A_2 (x) - A_1(x)|\leq 2\|Df\|_{\infty}w(x_2,x_3)$. Thus we see $|\partial_1g_w|$ is bounded independently of the values of $w$. To conclude \eqref{Guessed} it suffices to consider \eqref{TruEstimateDJ}. 
	\end{proof}
	
	By composing the above lemma with a translation and a rotation, we can produce a diffeomorphism from any 2-piece affine homeomorphism on $\er^3$, where the two pieces can necessarily be separated by a hyperplane. We now explain that the above lemma can be used to smoothen an $m$-piecewise affine homeomorphism $f$ which is affine on sectors in the sense described below.

	Let us have $m$ half-planes defined as
	$$
	F_i:=\{(s\cos\theta_i, s\sin\theta_i, t); s\in[0,\infty), t\in \er\}
	$$
	where $-\pi\leq\theta_1< \theta_2 < \dots <\theta_m=\theta_0<\pi$. Further let us have an $m$-piecewise affine homeomorphism $f$ such that $f =A_i$ on the sectors between the halfplanes $F_{i-1}$ and $F_i$ (where $F_0:= F_m$) and $A_i$ is affine, i.e.,
	\begin{equation}\label{ThisIsEf}
			f(r\cos \theta, r\sin \theta,z) = A_i(r\cos \theta, r\sin \theta,z)
	\end{equation}
	for $r>0, \theta\in [\theta_{i-1},\theta_i], z\in \er.$
	
	After a rotation $x,y$-plane  we may assume that $F_i = \{x_1 = 0\}$. Up to a reflection in  the image and preimage, we may assume that $0<[\partial_1A_i]^1\leq [\partial_1A_{i+1}]^1$ and apply Lemma~\ref{FaceLem}. Returning the reflections and rotations we get a map $g_{i}$ which is diffeomorphic on the set
	$$
		\{(r\cos\theta, r\sin\theta, z); r>r_i, \theta\in (\theta_{i-1},\theta_{i+1}), z\in \er\}.
	$$
	where $r_i>\|w_i\|_{\infty}$ is chosen so that
	$$
		\begin{aligned}
			G_i:=[F_i&+B(0,\|w_i\|)]\setminus [\er e_3+ B_{\er^3}(0,\tfrac{1}{4}r_i)]\\
			&\subset \{(t\cos\theta,t\sin\theta,z), t>0, \theta\in (\theta_{i-1}, \theta_{i+1}), z\in \er\}
		\end{aligned}
	$$
	and $\{G_i; i=1,\dots m\}$ are pairwise disjoint sets. For an illustration see Figure~\ref{Fig:Edges}. A precise condition, which gives us what we require from $r_i$ is
	\begin{equation}\label{SmallForUseOfEdges}
		\arctan\Big(\frac{\|w_i\|_{\infty}}{r_i}\Big)< \frac{\min\{|\theta_i - \theta_j + 2k\pi|; 1\leq i<j\leq m; \ k\in\mathbb{Z} \} \cup \{\tfrac{\pi}{8}\}}{8},
	\end{equation} 
	
	We call $r:=\max\{r_i; i=1,2,\dots, m\}$ and we define a map $\tilde{g}$ as follows
	\begin{equation}\label{CallMeGee}
	\begin{aligned}
		\tilde{g}(t\cos\theta, t\sin\theta,z)&:= g_i(t\cos\theta, t\sin\theta,z) \\
		\text{ for } t>\tfrac{1}{4}r, \theta \in [\tfrac{1}{2}\theta_{i-1}+ \tfrac{1}{2}&\theta_i,\tfrac{1}{2}\theta_{i}+ \tfrac{1}{2}\theta_{i+1}]+2k\pi \text{ and } k\in\mathbb{Z}. 
	\end{aligned}
	\end{equation}
	Since $g_i = f = g_{i+1}$ when $t>r$, $\theta\in [\theta_i +\arctan(\frac{\|w_i\|_{\infty}}{r}), \theta_{i+1} - \arctan(\frac{\|w_i\|_{\infty}}{r})]$, each $g_i$ is diffeomorhpic close to $F_i$ and $f$ is diffeomorphic away from $F_i$ we have that $g$ is a diffeomorphism of $\er^3\setminus [\er e_3 +B_{\er^3}(0,r)]$. It obviously follows that the orientation of $g$ is the same as $f$.
	
	\begin{figure}[h]
		\begin{tikzpicture}[line cap=round,line join=round,>=triangle 45,x=0.4cm,y=0.4cm]
			\clip(-10,-4.5) rectangle (12,9);
			\draw [line width=0.6pt,color=qqqqff,fill=qqqqff,fill opacity=0.25] (0.,0.) circle (1.6cm);
			\fill[line width=0.6pt,color=ffqqqq,fill=ffqqqq,fill opacity=0.10000000149011612] (149.5,150.5) -- (150.5,149.5) -- (3.072002435381516,2.5617964472241024) -- (2.56268793018166,3.071258792824408) -- cycle;
			\fill[line width=0.6pt,color=ffqqqq,fill=ffqqqq,fill opacity=0.10000000149011612] (3.9686269665968856,0.5) -- (150.,0.5) -- (150.,-0.5) -- (3.968626966596886,-0.5) -- cycle;
			\fill[line width=0.6pt,color=ffqqqq,fill=ffqqqq,fill opacity=0.10000000149011612] (-15.748412288273487,17.197222082334424) -- (-2.5846819051442904,3.0527724201485578) -- (-2.8731285502238464,2.7830077854523907) -- (-16.121661677959636,16.89252870299879) -- cycle;
			\fill[line width=0.6pt,color=ffqqqq,fill=ffqqqq,fill opacity=0.10000000149011612] (-11.512662771101233,18.28130045873067) -- (-11.03521922655396,18.561181157258382) -- (-1.9571078384183351,3.488513853892731) -- (-2.258337021934752,3.301501763646186) -- cycle;
			\draw [line width=0.6pt] (0.,0.)-- (150.,0.);
			\draw [line width=0.6pt] (0.,0.)-- (150.,150.);
			\draw [line width=0.6pt,domain=-13.440495362177069:0.0] plot(\x,{(-0.--8.655689107426667*\x)/-5.2900333129250585});
			\draw [line width=0.6pt,domain=-13.440495362177069:0.0] plot(\x,{(-0.--6.2213564577912015*\x)/-5.814091036110472});
			\draw [line width=0.6pt,color=ffqqqq] (149.5,150.5)-- (150.5,149.5);
			\draw [line width=0.6pt,color=ffqqqq] (150.5,149.5)-- (3.072002435381516,2.5617964472241024);
			\draw [line width=0.6pt,color=ffqqqq] (2.56268793018166,3.071258792824408)-- (149.5,150.5);
			\draw [line width=0.6pt,color=ffqqqq] (3.9686269665968856,0.5)-- (150.,0.5);
			\draw [line width=0.6pt,color=ffqqqq] (150.,0.5)-- (150.,-0.5);
			\draw [line width=0.6pt,color=ffqqqq] (150.,-0.5)-- (3.968626966596886,-0.5);
			\draw [line width=0.6pt,color=ffqqqq] (-15.748412288273487,17.197222082334424)-- (-2.5846819051442904,3.0527724201485578);
			\draw [line width=0.6pt,color=ffqqqq] (-2.8731285502238464,2.7830077854523907)-- (-16.121661677959636,16.89252870299879);
			\draw [line width=0.6pt,color=ffqqqq] (-16.121661677959636,16.89252870299879)-- (-15.748412288273487,17.197222082334424);
			\draw [line width=0.6pt,color=ffqqqq] (-11.512662771101233,18.28130045873067)-- (-11.03521922655396,18.561181157258382);
			\draw [line width=0.6pt,color=ffqqqq] (-11.03521922655396,18.561181157258382)-- (-1.9571078384183351,3.488513853892731);
			\draw [line width=0.6pt,color=ffqqqq] (-2.258337021934752,3.301501763646186)-- (-11.512662771101233,18.28130045873067);
			\draw (-2,5.5) node[anchor=north west] {$B_{\er^2}(0,r)$};
		\end{tikzpicture}
		\caption{The red areas are the sets where $G_i$. The blue disk represents $B_{\er^2}(0,r)\times\er$. If $w_i$ is small compared to $r$ then then $(G_i\cap G_j) \setminus (B_{\er^2}(0,r)\times\er) = \emptyset$ and in fact have positive distance from each other. This fact remains true as  long as the ratio $\frac{w_i}{r}$ remains small enough (i.e., is not affected by linear scaling). It holds that $g_i = g_{i+1} = f$ when $\theta\in (\theta_i,\theta_{i+1})$ and $(t\cos\theta_i, t\sin\theta_i,z)$ is outside $G_i$ (and outside $B_{\er^2}(0,r)\times\er$).}
	\end{figure}\label{Fig:Edges}
	
	\begin{lemma}\label{SegLem}
		Let \( r > 0 \), \( m \in \mathbb{N} \), and let \( -\pi \leq \theta_1 < \theta_2 < \dots < \theta_m = \theta_0 < \pi \). Let \( A_1, A_2, \dots, A_m : \mathbb{R}^3 \to \mathbb{R}^3 \) be linear maps such that
		$$
		f(t\cos\theta, t\sin\theta, z) = A_i(t\cos\theta, t\sin\theta, z), \quad \text{for } t > 0,\ \theta \in [\theta_i, \theta_{i+1}],\ z \in \mathbb{R},
		$$
		defines a sense-preserving homeomorphism.
		
		Further, let \( w_i \in \mathcal{C}^{\infty}(F_i) \) be positive functions, chosen small enough that \eqref{SmallForUseOfEdges} holds and \( |D w_i| < \sigma_i \), where \( \sigma_i > 0 \) are the constants from Lemma~\ref{FaceLem} (under the preceding conventions on rotations and reflections). Let \( \tilde{g} \) be the map defined in \eqref{CallMeGee}.
		
		Then there exists a constant \( \tilde{\sigma} > 0 \) such that if \( |D w_i| < \tilde{\sigma} \), there exists a diffeomorphism \( g : \mathbb{R}^3 \to \mathbb{R}^3 \) satisfying \( g = \tilde{g} \) on \( \mathbb{R}^3 \setminus (B_{\mathbb{R}^2}(0, r) \times \mathbb{R}) \). 
		
		Moreover, there exists a constant \( C_{\ref{EdgeEstimate}} > 0 \), independent of \( r \) and the specific choice of \( w_i \) (provided \eqref{SmallForUseOfEdges} holds), such that
		\begin{equation}\label{EdgeEstimate}
			\|Dg\|_{L^\infty} \leq C_{\ref{EdgeEstimate}} \left( \|Df\|_{L^\infty} + 1 \right).
		\end{equation}
	\end{lemma}

	\begin{proof}
		\step{1}{Setup}{EmptyHandle}
		
		We use the function $\eta$ from \eqref{Beef}. Up to an isometry in the image we may assume that $\partial_3f = (0,0,\lambda)$ for some $\lambda>0$, where $f$ is the piecewise affine map from \eqref{ThisIsEf} used to construct $\tilde{g}$. Recall that we denote $\pi_3(x_1,x_2,x_3) = (x_1,x_2,0)$ and $\tilde{\pi}_3(x_1,x_2,x_3) = (x_1,x_2)$.
		
		Let us outline the rough idea of how we proceed in the following. We start by defining $g$ on $S:=\{(0,0)\}\times\er$ as $f(0,0,x_3) = x_3\partial_3 f$. For brevity of notation we write $\Phi(t,\theta, x_3) = (t\cos\theta,t\sin\theta,x_3)$. In the following we define $g$ on concentric cylindrical annuli  around $S$ called $P_1, P_2,P_3$. On the outer annulus, $P_1$ we define $g$ as a convex combination which ``flattens''  the image of the annuli so that the image of planes perpendicular to $e_3$ are mapped to planes perpendicular to $\partial_3f$. The image of the outer boundary circle of $P_2 \cap \{x_3=c\}$ is a smooth star-shaped domain. On $P_2$ we do a radial squeezing so that the image of the inner boundary circle of $P_2 \cap \{x_3=c\}$ is a circle. On $P_3$ we use an isotopy between a diffeomorphism of $S^1$ and a rotation. Then on the remaining cylinder inside we define $g$ as the rotation from the previous step. We refer the reader to Figure~\ref{Fig:Idea} for an illustration of the idea.
		\begin{figure}[h]
			\includegraphics[width=12cm]{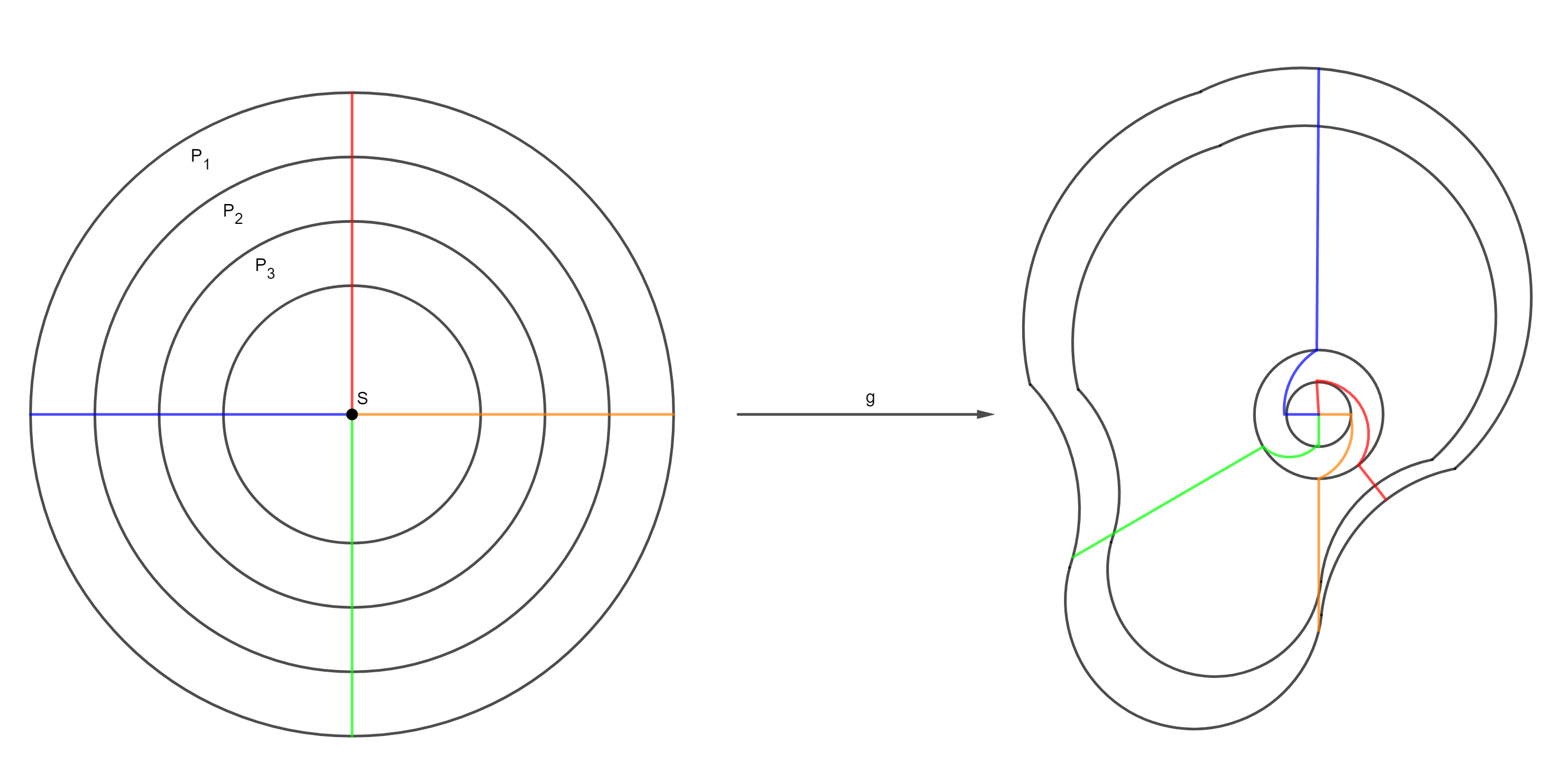}
			\caption{In the annulus $P_1$ we shift in the $e_3$ direction to so that hyperplanes disks inside $P_1$ are mapped onto hyperplanes. In $P_2$ we squeeze so that the image of a circle is a circle. In $P_3$ we `untwist' to achieve a rotation in the disks inside $P_3$ we extend as the rotation.}\label{Fig:Idea}
		\end{figure}
		
		\step{2}{Flatten images of hyperplanes pieces onto hyperplane pieces}{EmptyHandle}
		
		Let us start with the simplified case where $w_i$ are constant and satisfy \eqref{SmallForUseOfEdges}. On $P_1 := \{(t\cos\theta, t\sin\theta, x_3): x_3\in \er,t\in [\tfrac{4}{5}r,r], \theta\in \er\}$ we define
		$$	
		\begin{aligned}
			g(\Phi(t,\theta,x_3)) = & \Big[1-\eta\Big(\frac{5t-4r}{r}\Big)\Big] \Big((0,0,\lambda x_3) + \pi_3\circ \tilde{g}\big(\Phi(t,\theta,x_3)\big) \Big) \\
			&+ \eta\Big(\frac{5t-4r}{r}\Big)\tilde{g}\big(\Phi(t,\theta,x_3)\big).
		\end{aligned}
		$$
		Now we prove that the map $g$ is injective on $P_1$. Because $\partial_3f = (0,0,\lambda)$ on $\er^3$, it holds that $\partial_3\tilde{g} = (0,0,\lambda)$ on $\er^3\setminus  [B_{\er^2}(0,r/2)\times \er]$ (see \eqref{Sunlight}, we used the fact that $w_i$ are all constant). Then we easily also observe that $\partial_3g = (0,0,\lambda)$ everywhere on $P_1$. Since we may assume that $J_{\tilde{g}}>0$, we also have that
		$$
		\det\left(\begin{matrix}
			\partial_1\tilde{g}_1(x) &\partial_2\tilde{g}_1(x)\\
			\partial_1\tilde{g}_2(x) &\partial_2\tilde{g}_2(x)
		\end{matrix}\right) \geq C >0
		$$
		for some $C$ and all $x \in \er^3\setminus  [B_{\er^2}(0,r/4)\times \er]$. This implies also that
		$$
		\det\left(\begin{matrix}
			\partial_1{g}_1(x) &\partial_2{g}_1(x)\\
			\partial_1{g}_2(x) &\partial_2{g}_2(x)
		\end{matrix}\right)  \geq C >0
		$$
		on $P_1$. Then we may conclude that $J_g>0$ on $P_1$. It is not hard to observe that $g$ is injective on $\partial P_1$. In fact it suffices to prove that $g$ is injective on the the circle $\partial B_{\er^2}(0,4r/5)\times\{0\}$ which is easily observed from \eqref{D1Posi}. From the definition of $g$ we immediately see that $g$ is smooth on $P_1$ and has Jacobian bounded from below above zero and so is a diffeomorphism.
		
		\step{3}{Squeeze the smooth star-shaped domain radially into a small circle}{EmptyHandle}
		
		Call $P_2 := \Phi([\tfrac{3}{5}r,\tfrac{4}{5}r],\er,\er)$. We find a number $\rho>0$ so small that
		$$
		B_{\er^2}(0,\rho) \subset \er^2 \setminus \tilde{\pi}_3\circ \tilde{g}(\Phi([\tfrac{1}{2}r,\infty),\er,\er)).
		$$
		Then we define
		$$	
		\begin{aligned}
			g(\Phi(t,\theta,x_3)) = & (0,0,\lambda x_3)+\eta\Big(\frac{5t-3r}{r}\Big)\pi_3\circ \tilde{g}\big(\Phi(t,\theta,x_3)\big) \\
			&+ \Big(1-\eta\Big(\frac{5t-3r}{r}\Big)\Big)t\rho\frac{\pi_3\circ\tilde{g}\big(\Phi(\tfrac{3}{5}r,\theta,x_3)\big)}{|\pi_3\circ\tilde{g}\big(\Phi(\tfrac{3}{5}r,\theta,x_3)\big)|}.
		\end{aligned}
		$$
		The question of whether $g$ is injective is clearly only a 2-dimensional one, since  disjoint planes have disjoint images i.e.,
		$$
		g\big(\Phi([\tfrac{3}{5}r,\tfrac{4}{5}r],\er,a)\big) \cap g\big(\Phi([\tfrac{3}{5}r,\tfrac{4}{5}r],\er,b)\big) = \emptyset \quad \text{ if }a\neq b.
		$$
		It is easily observed that radial segments in $P_2$ are mapped to radial segments in the image with derivative bounded from below by $C(f)$ thanks to the choice of $\rho$. As stated in the previous step, \eqref{D1Posi} allows us to observe that $\partial_{\theta} g(\Phi(t,\theta, x_3))$ has a component in a direction perpendicular to $g(x)/|g(x)|$ bounded away from zero, i.e., also greater than $\rho C_f$. Therefore, we observe that
		$$
		\det\left(\begin{matrix}
			\partial_1{g}_1(x) &\partial_2{g}_1(x)\\
			\partial_1{g}_2(x) &\partial_2{g}_2(x)
		\end{matrix}\right)  \geq C >0
		$$
		on $P_2$, where $C$ depends on $f$ but not on values of $w$ or $r$ (we assume that $\frac{\rho}{r}$ is a fixed constant of the construction).

		\step{4}{`Untwist' the circle until it becomes a rotation and extend as the rotation close to the segment $S$}{EmptyHandle}
		
		It is well known that for any sense preserving diffeomorphism $h$ of the circle $S^1$ onto $S^1$ there exists an isotopy $\tilde{\alpha}: S^1\times[0,1] \to S^1$ with $\tilde{\alpha}(\zeta,1) = h(\zeta)$ and $\tilde{\alpha}(\zeta,0) = \zeta$ and that 
		\begin{equation}\label{SDRProp}
			C^{-1}\inf_{\zeta \in S^1}|h'(\zeta)|\leq |\partial_{1}\tilde{\alpha}(\tilde{\zeta},t)| \leq C\sup_{\zeta\in S^1}|h'(\zeta)|
		\end{equation}
		for all $t\in [0,1]$ and $\tilde{\zeta} \in S^1$. Further 
		\begin{equation}\label{Explicit}
			|\partial_{t}\tilde{\alpha}(\zeta,t)|\leq C_{\ref{Explicit}}
		\end{equation}
		for all $t\in [0,1]$ and $\tilde{\zeta} \in S^1$. A proof and explicit construction of this can be found in \cite[Section~3.1]{MP} but it is not hard to construct the map by hand using a smooth convex combination in polar coordinates. For us
		$$
		h(\cos\theta, \sin\theta) = \frac{\tilde{\pi}_3\circ\tilde{g}\big(\Phi(\tfrac{3}{5}r,\theta,x_3)\big)}{|\tilde{\pi}_3\circ\tilde{g}\big(\Phi(\tfrac{3}{5}r,\theta,x_3)\big)|} = \frac{\tilde{\pi}_3\circ g(\Phi(\tfrac{3}{5}r,\theta,x_3))}{|\tilde{\pi}_3\circ g(\Phi(\tfrac{3}{5}r,\theta,x_3))|}
		$$
		and so let $\alpha$ be the corresponding isotopy.

		Call $P_3 := \Phi([\tfrac{2}{5}r,\tfrac{3}{5}r],\er,\er)$. We define 
		$$
		\begin{aligned}
			g(\Phi(t,\theta,x_3)) = & (0,0,\lambda x_3) +t\rho \alpha\Big(\theta,\frac{5t-2r}{r}\Big).
		\end{aligned}
		$$
		Since $\alpha$ is a diffeomorphism, we have that $g$ is injective on $P_3$. Thus, we observe that $g$ is injective on $P_1\cup P_2\cup P_3$. Further, because $\alpha$ is smooth, $g$ is smooth and can be smoothly extended as 
		$$
		\begin{aligned}
			g(\Phi(t,\theta,x_3)) = & (0,0,\lambda x_3) + t\rho\alpha(\theta, 0)
		\end{aligned}
		$$
		for $t\in [0.\tfrac{r}{4}]$. Notice that this agrees with our initial definition of $g$ on $S$.
		
		By considering the definition of $g$ on each part $P_1,P_2,P_3$ separately, we observe that
		\begin{equation}\label{SegLemDerEst}
			\begin{aligned}
				\partial_{3} &g(\Phi(x_1,t,\phi)) = (0,0,\lambda) \leq \|Df\|_{L^\infty},\\
				|\partial_t &g(\Phi(t,\theta, x_3))| \leq  C\frac{1}{r}\sup_{\theta\in \er}|\pi_3\circ\tilde{g}\big(\Phi(r,\theta, 0)\big)| \\
				&\ \ +C\frac{1}{r}\sup_{\Phi(t,\theta,0 )\in P_1} \Big|\Big[\tilde{g}\big(\Phi(t,\theta, 0)\big) - \pi_3\circ\tilde{g}\big(\Phi(t,\theta, 0)\big) \Big]\Big|\\
				&\ \ + \frac{\rho}{r}C_{\ref{Explicit}}\\
				&\leq C\|Df\|_{L^\infty}, \text{ and}\\
				|\partial_v &g(\Phi(t,\theta, x_3))| \leq C\|Df\|_{L^\infty}
			\end{aligned}
		\end{equation}
		where by $v = v(\theta)$ we denote the vector $(\sin\theta, -\cos\theta)$. These estimates suffice to give \eqref{EdgeEstimate} since we the ratio $\rho/r$ is a fixed constant depending on $f$. We have proved also that $J_g\geq C>0$ implying that, thanks to Lemma~\ref{Book}, $g$ is a diffeomorphism since we know that $g$ is injective far away from $S$ by Lemma~\ref{FaceLem}.
		
		\step{5}{The construction works also for non-constant $w_i$}{EmptyHandle}
		
		Let us now assume that the functions $w_i$ are not constant. The derivative $Dg$ depends continuously on $Dw_i$ and so we find  a $\tilde{\sigma}>0$ so small that the Jacobian of $g$ remains bounded away from 0 as soon as $|Dw|<\tilde{\sigma}$. The conclusion now mirrors that above for constant $w_i$.
	\end{proof}
	
	Let us now prove that the construction from Lemma~\ref{SegLem} can be modified to allow us to modify the parameter $r$ with respect to $x_3$ while still giving a diffeomorphic extension and the same estimates on $\|Dg\|_{L^\infty}$. 
	
	\begin{corollary}\label{SegFix}
		Let $f$ be the piecewise affine map from \eqref{ThisIsEf}. There exists a $\xi>0$ such that the following holds. Let $r\in \mathcal{C}^{\infty}(\er)$ be a positive function such that $|r'|\leq \xi$ and let $w_i\in \mathcal{C}^{\infty}(F_i)$ satisfy $\|Dw_i\|_{L^\infty}\leq \min\{\sigma_i, \tilde{\sigma}\}$ (the numbers from Lemma~\ref{FaceLem} and Lemma~\ref{SegLem}), let also $w_i$ satisfy the following version of \eqref{SmallForUseOfEdges}
		\begin{equation}\label{SmallForUseOfEdges2}
			\arctan\Big(\frac{w_i(x_1,x_2,x_3)}{r(x_3)}\Big)< \frac{\min\{|\theta_i - \theta_j + 2k\pi| \} \cup \{\tfrac{\pi}{8}\}}{8}
		\end{equation}
		for all $(x_1,x_2,x_3)\in \er^3$ such that $x_1^2+x_2^2 \geq r^2(x_3)/16$. Then the map $\tilde{g}$ from  \eqref{CallMeGee} is a diffeomorphism of $\er^3\setminus \{x\in \er^3; x^2_1+x^2_2\leq r^2(x_3)/4\}$ and by construction in Lemma~\ref{SegLem} we get a diffeomorphism $g$ equaling $\tilde{g}$ on $\er^3\setminus \{x\in \er^3; x^2_1+x^2_2\leq r^2(x_3)\}$ and satisfying \eqref{EdgeEstimate}.
	\end{corollary}
	\begin{proof}
		It suffices to observe that $Dg$ depends continuously on $r'$ and the estimates are independent of $r$ (they depend on the ratio $\frac{\rho}{r}$ which we keep constant).
	\end{proof}
	
	\begin{lemma}\label{VertLem}
		Let $R>0$ and let $T_1,T_2,\dots T_k$ be 3-simplexes such that $B_{\er^3}(0,2R)\subset\bigcup_{i=1}^kT_i =:G$ and $(0,0,0)$ is a common vertex of all $T_i$. Let $A_i$ be linear maps and
		$$
		f(x) = A_i(x)\quad x\in T_i
		$$
		be a sense-preserving piecewise affine homeomorphism.
		
		Let the 2-subsimplexes of $T_1, \dots, T_k$ containing $(0,0,0)$ be indexed as $F_1,F_2,\dots, F_n$ and the 1-subsimplexes containing $(0,0,0)$ be indexed as $S_1,S_2,\dots, S_m$.
		
		Then there exists an $r_0 = c(\{S_i\}_{i=1}^m)R$ such that the sets $[S_i+B(0,r_0)]\setminus B(0,R/2)$ are pairwise disjoint. For each $S_i$, $i=1,\dots, m$, let us have numbers  $0<r_i \leq r_0$ and for each $F_j$ let us have numbers $w_j$ so small, that (after an appropriate rotation in the preimage) they satisfy \eqref{SmallForUseOfEdges}.
		
		Let $\hat{g} = \hat{g}_{\{r_i\}_{i=1}^m, \{w_j\}_{j=1}^n}$ be the diffeomorphism of $B(0,2R)\setminus B(0,R/2)$ constructed using: - Lemma~\ref{FaceLem} in neighborhoods of each $ F_j \setminus [S_i + B(0, r_i)]$, and
		- Lemma~\ref{SegLem} in neighborhoods of each $ S_i$.
		
		Then there exists an $\alpha, \beta>0$ such that if
		$$
		\frac{w_j}{R} \leq \alpha \quad \text{and} \quad \frac{r_i}{R}\leq \beta,
		$$
		there exists a diffeomorphism $ g: B(0,2R) \to \mathbb{R}^3$ such that
		$$
		g = \hat{g} \quad \text{on } B(0,2R) \setminus B(0,R).
		$$
	\end{lemma}
	
	\begin{proof}
		\step{1}{Projection to the sphere is injective}{arg3a}
		Let us prove that the map $x \to \frac{\hat{g}(x)}{|\hat{g}(x)|}$ is injective on $\partial B(0,t)$ for $t\in [R/2,2R]$.  The claim is easy to observe when we replace $\hat{g}$ with $f$ since $f$ is a piecewise affine homeomorphism. This means (since we assume that $f$ is sense-preserving) that $\det Df\geq C>0$, a.e. on $G$. For each point $x$ in the preimage we consider an orthonormal positively-oriented basis of $\er^3$ $\{\frac{x}{|x|},u_2,u_3\}$. For each point $f(x)$ in the image we consider an orthonormal positively-oriented basis of $\er^3$ $\{\frac{f(x)}{|f(x)|},v_2,v_3\}$. We calculate $J_f$ with respect to this basis. Since $\partial_{\frac{x}{|x|}}f = \frac{f(x)}{|x|}$ we get that the subdeterminant of $Df$ with respect to $\{u_2,u_3\}\times\{v_2,v_3\}$ is necessarily positive and by linearity is independent of $|x|$. This implies that the map $x\to \frac{f(tx)}{|f(tx)|}$ defined on $\partial B(0,1)$ is a positively oriented bi-Lipschitz map of the unit sphere onto itself (and is independent of the parameter $t$).

		We want to prove the same estimate for $\hat{g}$ as we have for $f$, i.e., the subdeterminant of $D\hat{g}$ with respect to $\{u_2,u_3\}\times\{v_2,v_3\}$ is positive for every $t$. We may assume that $J_{\hat{g}} \geq C>0$ on $B(0,2R)\setminus B(0,R/2)$. On most of $\partial B(0,t)$ we are far away from any $F_j$ and so $\hat{g} = f$ (see \eqref{NoChange}) and therefore we have the estimate immediately. Let us assume that we are in a $w_j$ neighbourhood of $F_j$ (but not close to any $S_i$). By basic geometry and \eqref{Guessed} (since we have $w_j$ constant we may assume that $\sigma = 0$), we have that 
		$$
		\Big|\partial_{\frac{x}{|x|}}\hat{g} - \partial_{\frac{\pi_{F_j}(x)}{|(x)|}}\hat{g}\Big|  = \Big|\partial_{\frac{x - \pi_{F_j}(x)}{|x|}}\hat{g} \Big|\leq \frac{w_j}{R}|\partial_v\hat{g}| \leq C\frac{w_j}{R}\|Df\|_{L^\infty},
		$$
		where $\pi_{F_j}$ is the orthogonal projection of $\er^3$ onto $F_j$ and $v$ is a unit vector perpendicular to $F_j$. We calculate $J_{\hat{g}}$ by an expansion of the row corresponding to the component parallel to $\hat{g}(x)$ as
		$$
		J_{\hat{g}} \leq |\partial_{\frac{x}{|x|}}\hat{g}|\Lambda +\big|\partial_{\frac{x}{|x|}}\hat{g} - \partial_{\frac{\pi_{F_j}(x)}{|(x)|}}\hat{g}\big|\|Df\|_{L^\infty}^2
		$$
		where by $\Lambda$ we denote the subdeterminant of $D\hat{g}$ with respect to $\{u_2,u_3\}\times\{v_2,v_3\}$. This means that
		$$
		\Lambda \geq |\partial_{\frac{x}{|x|}}\hat{g}|^{-1}\Big( J_{\hat{g}} - C\frac{w_j}{R}\|Df\|_{L^\infty}^3\Big).
		$$
		Assuming that the ratio $\frac{w_j}{R}$ is sufficiently small we have that $\Lambda$ is positive bounded away from zero.
		
		Using the same argument on neighborhoods of $S_i$ we observe that also here $\Lambda$ is positive bounded away from zero as soon as $\frac{r_i}{R}$ is sufficiently small. This implies that $\frac{\hat{g}}{|\hat{g}|}$ is a local homeomorphism of $\partial B(0,t)$ onto $\partial B(0,1)$, for $t\in [R/2,2R]$ and therefore, by Theorem~\ref{AlgTop}, $\hat{g}$ is a homeomorphism. 
		
		\step{2}{An injective convex combination of $\hat{g}$ and its projection onto a sphere}{arg3}
		It follows from step~\ref{arg3a} that the image of $B(0,\frac{3}{4}R)$ under $\hat{g}$ is a star-shaped domain containing the origin and having smooth boundary. This means we find a $\rho=C(f)r>0$ such that $B(0,2\rho) \subset\hat{g}(B(0,\frac{3}{4}R))$. We define $g$ on $B(0,R)\setminus B(0,\frac{3}{4}R)$ as follows
		$$
		g(x) = \eta\Big(\frac{4|x|-3r}{r}\Big)\hat{g}(x) + \Big(1- \eta\Big(\frac{4|x|-3r}{r}\Big)\Big)\rho|x|\frac{\hat{g}(x)}{|\hat{g}(x)|}.
		$$
		The injectivity of this map follows from the injectivity proved in the previous step and the choice of $\rho$, which guarantees that
		$$
		\Big\langle \partial_{\frac{x}{|x|}}\hat{g}, \frac{\hat{g}(x)}{|\hat{g}(x)|} \Big\rangle> 0.
		$$
		It follows that $J_g$ is bounded from below away from zero.
		
		\step{3}{Reduce to a linear map on $B(0,\tfrac{3}{4}R)\setminus B(0,\tfrac{1}{2}R)$}{arg3}
		We have that the map $\frac{4}{3\rho}g(\frac{4}{3R}x)$ is a sense-preserving diffeomorphism of $S^2$ onto $S^2$. Then we may apply Theorem~\ref{Smale}, to get a isotopy $\Psi:S^2\times[0,1] \to S^2$ such that $\Psi(\cdot, t)$ is a diffeomorphism of $S^2$ onto $S^2$, $\Psi(x, 1) = \frac{4}{3\rho}g(\frac{4}{3R}x)$ and $\Psi(x, 1) = x$. Then we may define
		$$
		g(x):= \rho|x|\Psi\Big(x,\eta\Big(\frac{4|x|-2R}{R}\Big)\Big)
		$$
		for $x\in B(0,\tfrac{3}{4}R)$. One may easily observe that $g(x) = \rho x$ on $B(0,R/2)$. It is easy to observe that $g$ is now a diffeomorphic extension of $\hat{g}$ restricted to $B(0,2R)\setminus B(0,R)$.
	\end{proof}
	
	\subsection{Proof of Theorem~\ref{Core}}
	Without loss of generality we may assume that $f$ is sense preserving (if not consider a reflection).
	
	Let us first conduct the proof in a simplified case, i.e., where $f$ is a \textit{finitely} piecewise affine homeomorhpism. We are able to resolve this case by choosing a constant $w_j$ for each 2-subsimplex $F_j$ and  a constant $r_i$ for each 1-subsimplex $S_i$. Let the vertexes of the simplicial complex, i.e., its 0-subsimplexes  be indexed by $V_k$. We start by choosing a preliminary $\tilde{R_k}$  for each $V_k$ satisfying only the requirement that $B(V_k, \tilde{R}_k)$ are pairwise disjoint and that if $T_m\cap B(V_k, \tilde{R}_k) \neq \emptyset$, then $V_k\in T_m$ for each simplex $T_m$. Now we may choose $\tilde{r}_i,\tilde{w}_j$ so small that the hypothesis of Lemma~\ref{SegLem} and Lemma~\ref{VertLem} are satisfied. Now we use Lemma~\ref{FaceLem}, Lemma~\ref{SegLem} and Lemma~\ref{VertLem} to construct a diffeomorphism $\tilde{g}$ equaling $f$ whenever we are outside the set
	$$
	E_1:=\bigcup_kB(V_k,\tilde{R}_k) \cup \bigcup_i(S_i + B(0,\tilde{r}_i)) \cup \bigcup_j(F_j + B(0,\tilde{w}_j)).
	$$
	We need to refine this map, since we have no estimate on $|D\tilde{g}|$ on $B(V_k, \tilde{R}_k)$. Let us denote $M:= \|D\tilde{g}\|_{L^\infty}$. We easily observe that $M$ is a function of the ratios $\frac{r_i}{R}$, $\frac{w_j}{R}$ and $\frac{w_j}{r_i}$, i.e., if we scale all of these simultaneously by some $\lambda \in (0,1]$ we get a diffeomorphism $g_{\lambda}$ each time also satisfying the estimate 
	\begin{equation}\label{Napoli}
		\|Dg_{\lambda}\|_{L^{\infty}}= M.
	\end{equation}
	Let us denote by $K:= \|D\tilde{g}^{-1}\|_{L^\infty}$, arguing in the same way we have:
	\begin{equation}\label{Napoli1}
		\|Dg^{-1}_{\lambda}\|_{L^{\infty}}= K.
	\end{equation}
	
	On the other hand the measure of the set $E_{\lambda} = \{g_{\lambda}\neq f\}$ tends to zero as $\lambda \to 0^+$. By this argument we estimate that
	$$
	\|Df-Dg_{\lambda}\|_{L^p}^p\leq M^p\mathcal{L}^3(E_{\lambda}),
	$$
	for each $p\in[1,\infty)$.
	
	We may estimate the error for the inverse in the same way, i.e., by the uniform Lipschitz quality of all $g_{\lambda}$ we have that 
	$$
	\mathcal{L}^3(g_{\lambda}(E_{\lambda}))\leq M^3\mathcal{L}^3(E_{\lambda}).
	$$
	Further because $|Dg^{-1}_{\lambda}|$ is bounded independently of $\lambda$ we have that 
	$$
	\|Df^{-1}-Dg^{-1}_{\lambda}\|_{L^p} \to 0 \text{ as } \lambda \to 0 
	$$
	for each $p\in[1,\infty)$.
	
	Now we consider the general case where we have a locally finite simplicial complex. Let us choose an increasing sets sequence $A_1\subset A_2\subset \dots$, such that each $A_n$ is the finite union of simplexes $\{T_k\}_{k=1}^{N_n}$ and such that if $T_{k}\cap \partial A_n \neq \emptyset$ then $T_k\cap \partial A_{n+1} = \emptyset$. By $n$-border simplexes we mean those simplexes in $A_{n+1}\setminus A_n$ such that $T_{k}\cap \partial A_n \neq \emptyset$. By $H_n$ we denote the set of simplexes in $A_n\setminus A_{n-1}$ and the $n$-border simplexes. For each $n$ we choose a value $\tilde{\lambda}_n$ so small that whenever we construct $\tilde{g}_n$ with $0<\lambda_n\leq \tilde{\lambda}_n$using the above process, we have
	$$
	\|f - \tilde{g}_n\|_{W^{1,p}} + \|f^{-1} - \tilde{g}^{-1}_n\|_{W^{1,q}}< 2^{-n}\epsilon,
	$$
	where the norms are taken over the interior of the union of simplexes of $H_n$ and the intersection of their images in $f$ and $\tilde{g}_n$ respectively.
	
	In order to connect up the various $\tilde{g}_n$ into one diffeomorphism there is one more step that we need to conduct, namely on $n$-border simplexes we have defined $\tilde{g}_n$ and $\tilde{g}_{n+1}$. The idea is to allow $w_j$ and $r_i$ to smoothly change between the two values on the $F_j$ and $S_i$ respectively. On the other hand, if $F_j$ or $S_i$ are too short  this may not be possible since in order for Lemma~\ref{FaceLem}, Lemma~\ref{SegLem} and Corollary~\ref{SegFix} to hold we need $|Dw_j|$ and $|r'|$ to be very small. On the other hand, now knowing the length of the segments $S_i$ we may now choose (assume $\tilde{\tilde{\lambda}}_2:= \tilde{\lambda}_2$) $\lambda_n< \tilde{\tilde{\lambda}}_n$ and $\tilde{\tilde{\lambda}}_{n+1}< \tilde{\lambda}_{n+1}$ so that the difference in respective values (for any choice of $\lambda_{n+1}<\tilde{\tilde{\lambda}}_{n+1}$) of $w_j^n, w_j^{n+1}$ and of $r_i^n, r_i^{n+1}$ are so small that there exist smooth functions $w_j$ and $r_i$ having the chosen values for $H_n$ near $A_n$ and having the values chosen for $H_{n+1}$ further away from $A_n$, while $|Dw_i|$ and $|r'|$ remain as small as required.
	
	By the preceding lemmas (and Corollary~\ref{SegFix}) we construct a diffeomorphism $g$. By the choices of $\lambda_n<\tilde{\lambda}_n$ we have that
	$$
	\|f-g\|_{W^{1,p}} +  \|f^{-1} - g^{-1}\|_{W^{1,q}}< \epsilon,
	$$ 
	where the norms are taken over $\Omega$ and $f(\Omega)$ respectively.
	
	\qed

	\subsection{Proof of Theorem~\ref{main}}	
	
	By \eqref{Napoli} and \eqref{Napoli1}, for each $f$, there exists a $C(f)$ such that
	$$
	\|Df-Dg_n\|_{L^{\infty}(H_n)} + \|Df^{-1} - Dg_n^{-1}\|_{L^{\infty}(f(H_n))} \leq C(f).
	$$
	Then using the fact that $\mathcal{L}^3(E_{\lambda_n, n}) \to 0$ as $\lambda_n \to 0$, where $E_{\lambda_n} = H_n\cap \{g_{\lambda_n,n}\neq f\}$, $f$ is bi-Lipschitz on $H_n$ and Lemma~\ref{Rozumny} we see that
	$$
	\|Df-Dg_n\|_{X(H_n)} + \|Df^{-1} - Dg_n^{-1}\|_{Y(f(H_n))} \to 0
	$$
	as $\lambda_n\to 0$. The estimates for the function values are even easier.
	
	\section{Remark on the proof for higher dimensions}
	The proof in dimensions $d=4$ follows almost exactly the same way as the case in $d=3$, since the arguments in Lemma~\ref{FaceLem} are independent on the number of dimensions of the hyperplane $\{x_1 = 0\}$ and the proof of Lemma~\ref{SegLem} does not change when we replace the segment in $\er^3$ with a 2-dimensional plane in $\er^4$. The only notable difference is that instead of Theorem~\ref{Smale} we use the following generalization:
	\begin{thm}\label{Tomasova}
		For $d=2, 3, 4, 6, 12, 56, 61$ and every $\alpha_0$, a sense-preserving $\mathcal{C}^{\infty}$-diffeomorphism of $S^{d-1} = \partial B_{\er^d}(0,1)$ onto itself there exists an $\alpha\in \mathcal{C}^{\infty}([0,1]\times S^{d-1}, S^{d-1})$ such that $\alpha(t,\cdot)$ is a diffeomorphism of $S^{d-1}$ for each $t$, $\alpha(0,\cdot) = \alpha_0(\cdot)$ and $\alpha(1,\cdot) = \id(\cdot)$. 
	\end{thm}	
	This theorem is equivalent to stating that the group of sense-preserving diffeomorphisms of the sphere $S^{d-1}$ is connected for $d=2, 3, 4, 6, 12, 56, 61.$ As mentioned above for $d=3$ the theorem was proven by Smale in \cite[Theorem A]{S}. For $d=4$ it was proven by Hatcher in \cite{H} and the cases of $d \geq 6$ can be found in \cite[Section II]{Hatcher50} and \cite[Table 1 and 2]{StableHomotopyGroups}, where it can also be found that the theorem fails for all other $d \leq 90$ outside $d=5.$ It is still an open problem for $d=5$ and whether it holds for any $d>90.$
	
	Since the method of our proof is contingent on finding such an isotopy $\alpha$, the obstacle for generalizing to higher dimension than $d=4$, is that we do not know that the space of sense-preserving diffeomorphisms from the sphere onto itself is connected for dimension $5$ and in fact is false for dimensions other than the ones mentioned in Theorem~\ref{Tomasova}.
	
	\subsection*{Acknowledgement}
	The authors would like to thank the anonymous referee for reading the manuscript with great care and interest. We are especially grateful for their insightful comments and corrections, which helped eliminate errors and improve the accuracy and clarity of the text.

\end{document}